\documentclass[
reqno]{amsart}
\usepackage{latexsym,amsmath,amssymb,amscd}
\usepackage[all]{xy}
\usepackage{verbatim,color}

\def\today{\ifcase \month \or
   January \or February \or March \or April \or
   May \or June \or July \or August \or
   September \or October \or November \or December \fi
   \space\number\day , \number\year}

\makeatletter
  \newcommand\@dotsep{4.5}
  \def\@tocline#1#2#3#4#5#6#7{\relax
     \ifnum #1>\c@tocdepth 
     \else
     \par \addpenalty\@secpenalty\addvspace{#2}%
     \begingroup \hyphenpenalty\@M
     \@ifempty{#4}{%
     \@tempdima\csname r@tocindent\number#1\endcsname\relax
        }{%
         \@tempdima#4\relax
           }%
      \parindent\z@ \leftskip#3\relax \advance\leftskip\@tempdima\relax
      \rightskip\@pnumwidth plus1em \parfillskip-\@pnumwidth
       #5\leavevmode\hskip-\@tempdima #6\relax
       \leaders\hbox{$\m@th
       \mkern \@dotsep mu\hbox{.}\mkern \@dotsep mu$}\hfill
       \hbox to\@pnumwidth{\@tocpagenum{#7}}\par
       \nobreak
        \endgroup
         \fi}
\makeatother

\begin{document}

\makeatletter
\@addtoreset{figure}{section}
\def\thefigure{\thesection.\@arabic\c@figure}
\def\fps@figure{h,t}
\@addtoreset{table}{bsection}

\def\thetable{\thesection.\@arabic\c@table}
\def\fps@table{h, t}
\@addtoreset{equation}{section}
\def\theequation{
\arabic{equation}}
\makeatother

\newcommand{\bfi}{\bfseries\itshape}

\newtheorem{theorem}{Theorem}
\newtheorem{acknowledgment}[theorem]{Acknowledgment}
\newtheorem{corollary}[theorem]{Corollary}
\newtheorem{definition}[theorem]{Definition}
\newtheorem{example}[theorem]{Example}
\newtheorem{lemma}[theorem]{Lemma}
\newtheorem{notation}[theorem]{Notation}
\newtheorem{problem}[theorem]{Problem}
\newtheorem{proposition}[theorem]{Proposition}
\newtheorem{question}[theorem]{Question}
\newtheorem{remark}[theorem]{Remark}

\numberwithin{theorem}{section}
\numberwithin{equation}{section}

\newcommand{\todo}[1]{\vspace{5 mm}\par \noindent
\framebox{\begin{minipage}[c]{0.85 \textwidth}
\tt #1 \end{minipage}}\vspace{5 mm}\par}

\renewcommand{\1}{{\bf 1}}

\newcommand{\Ad}{{\rm Ad}}
\newcommand{\ad}{{\rm ad}}

\newcommand{\Cartan}{{\rm Cartan}}
\newcommand{\de}{{\rm d}}
\newcommand{\Decomp}{{\rm Decomp}}
\newcommand{\diag}{{\rm diag}}
\newcommand{\fin}{{\rm fin}}
\newcommand{\gl}{{{\mathfrak g}{\mathfrak l}}}
\newcommand{\id}{{\rm id}}
\newcommand{\ie}{{\rm i}}
\newcommand{\Gr}{{\rm Gr}}
\newcommand{\GL}{{\rm GL}}
\renewcommand{\Im}{{\rm Im}}
\newcommand{\Ker}{{\rm Ker}\,}
\newcommand{\Lie}{\textbf{L}}
\newcommand{\OO}{{\rm O}}
\newcommand{\Ran}{{\rm Ran}\,}
\newcommand{\rank}{{\rm rank}\,}
\newcommand{\sa}{{\rm sa}}
\newcommand{\Sp}{{\rm Sp}}
\renewcommand{\sp}{{{\mathfrak s}{\mathfrak p}}}
\newcommand{\spec}{{\rm spec}}

\newcommand{\Tr}{{\rm Tr}\,}

\newcommand{\CC}{{\mathbb C}}
\newcommand{\KK}{{\mathbb K}}
\newcommand{\NN}{{\mathbb N}}
\newcommand{\RR}{{\mathbb R}}
\renewcommand{\SS}{{\mathbb S}}
\newcommand{\TT}{{\mathbb T}}

\newcommand{\bb}{b}
\newcommand{\vv}{v}
\newcommand{\ww}{w}

\newcommand{\G}{{\rm G}}
\newcommand{\U}{{\rm U}}
\newcommand{\Gl}{{\rm GL}}
\newcommand{\SL}{{\rm SL}}
\newcommand{\SU}{{\rm SU}}
\newcommand{\VB}{{\rm VB}}

\newcommand{\Ac}{{\mathcal A}}
\newcommand{\Bc}{{\mathcal B}}
\newcommand{\Cc}{{\mathcal C}}
\newcommand{\Dc}{{\mathcal D}}
\newcommand{\Fc}{{\mathcal F}}
\newcommand{\Hc}{{\mathcal H}}
\newcommand{\Ic}{{\mathcal I}}
\newcommand{\Jc}{{\mathcal J}}
\newcommand{\Kc}{{\mathcal K}}
\newcommand{\Pc}{{\mathcal P}}
\newcommand{\Qc}{{\mathcal Q}}
\newcommand{\Sc}{{\mathcal S}}
\newcommand{\Tc}{{\mathcal T}}
\newcommand{\Vc}{{\mathcal V}}
\newcommand{\Bg}{{\mathfrak B}}
\newcommand{\ag}{{\mathfrak a}}
\newcommand{\bg}{{\mathfrak b}}
\renewcommand{\gg}{{\mathfrak g}}
\newcommand{\hg}{{\mathfrak h}}
\newcommand{\mg}{{\mathfrak m}}
\newcommand{\og}{{\mathfrak o}}
\newcommand{\ug}{{\mathfrak u}}
\newcommand{\nng}{{\mathfrak n}}
\newcommand{\pg}{{\mathfrak p}}
\newcommand{\Ig}{{\mathcal I}}
\newcommand{\Jg}{{\mathcal I}}
\newcommand{\Lg}{{\mathfrak L}}
\newcommand{\Sg}{{\mathfrak S}}

\markboth{}{}

\makeatletter
\title[Interplay between Algebraic Groups 
and Operator Ideals]
{Interplay between Algebraic Groups, Lie Algebras and Operator Ideals}
\author[Daniel Belti\c t\u a]{Daniel Belti\c t\u a$^*$}
\thanks{$^*$
Partially supported by the Grant
of the Romanian National Authority for Scientific Research, CNCS-UEFISCDI,
project number PN-II-ID-PCE-2011-3-0131, 
and by Project MTM2010-16679, DGI-FEDER, of the MCYT, Spain.}
\author[Sasmita Patnaik]{Sasmita Patnaik$^{**}$}
\thanks{$^{**}$Partially supported by the Graduate Dissertation Fellowship 
from the Charles Phelps Taft Research Center.}
\author[Gary Weiss]{Gary Weiss$^{***}$}
\thanks{$^{***}$
Partially supported by 
Simons Foundation Collaboration Grant for Mathematicians \#245014.}
\dedicatory{Dedicated to the memory of Mih\'aly Bakonyi}

\makeatother

\begin{abstract}
In the framework of operator theory, 
we investigate a close Lie theoretic relationship between all operator ideals 
and certain classical groups of invertible operators that can be described as 
the solution sets of certain algebraic equations, 
hence can be regarded as infinite-dimensional linear algebraic groups. 
Historically, this has already been done for only the complete-norm ideals;  
in that case one can work within the framework of the well-known Lie theory for Banach-Lie groups.  
That kind of Lie theory is not applicable for arbitrary operator ideals, 
so we needed to find a new approach for dealing with the general situation. 
The simplest instance of the aforementioned relationship is provided by 
the Lie algebra $\ug_{\Ic}(\Hc)=\{X\in\Ic\mid X^*=-X\}$ 
associated with the group $\U_{\Ic}(\Hc)=\U(\Hc)\cap(\1+\Ic)$ where $\Ic$ is an arbitrary operator ideal 
in $\Bc(\Hc)$ and $\U(\Hc)$ is the full group of unitary operators. 
We investigate the Cartan subalgebras (maximal abelian self-adjoint subalgebras) of $\ug_{\Ic}(\Hc)$ for 
$\{0\}\subsetneqq\Ic\subsetneqq\Bc(\Hc)$, 
and obtain an uncountably many $\U_{\Ic}(\Hc)$-conjugacy classes of these Cartan subalgebras. 
The cardinality proof will be given in a follow up paper \cite{BPW13} 
and stands in contrast to the $\U(\Hc)$-uniqueness work of de la Harpe \cite{dlH72}.
\end{abstract}

\maketitle

\tableofcontents

\newpage

\section{Introduction}

In the framework of operator theory on
Hilbert spaces, 
in this paper we investigate  a rather close Lie theoretic relationship 
between all operator ideals and certain so-called classical groups of invertible operators 
that can be described as the sets of solutions of certain algebraic equations, 
e.g., $T^{-1}=\widetilde{J}T^*\widetilde{J}^{-1}$
(see Definitions \ref{complex_gr} and \ref{real_gr} below)
hence can be regarded as infinite-dimensional linear algebraic groups.

In the study of complete normed operator ideals, 
the aforementioned classical groups of invertible operators 
have natural Banach-Lie group structure. 
These groups and their Lie algebras were developed systematically in \cite{dlH72} 
and they go back to \cite{Sch60} and \cite{Sch61} in the case of the Hilbert-Schmidt ideal. 
More recently, the study of these infinite-dimensional Lie groups was pursued to investigate 
various aspects of their structure theory: 
Cartan subalgebras (\cite{Al82}), classification (\cite{CGM90}), 
topological properties (\cite{Ne02}), 
coadjoint orbits (\cite{Ne04}), Iwasawa decompositions (\cite{Be09} and \cite{Be11}), etc. 
Perhaps even more remarkable,  
the representation theory of some classical groups associated with the Schatten ideals  
was developed using the model of the complex semisimple Lie groups 
(see for instance \cite{SV75},  \cite{Bo80}, \cite{Ne98}, and their references). 

On the other hand, there are many interesting operator ideals that do not admit any complete norm 
(see for instance \cite{Va89} and \cite{DFWW04}, and also Example~\ref{ideals_ex} below). 
The classical groups associated with these ideals 
can still be defined via algebraic equations, and yet these groups are no longer Lie groups. 
It seems natural, therefore, to study the aforementioned Lie theoretic relationship 
beyond the realm of complete normed ideals. 
We are thus led to develop a suitable Lie theory for 
linear algebraic groups in infinite dimensions   
by extending the earlier approach of \cite{HK77} 
to the classical groups associated with arbitrary  (i.e., possibly non-normed) operator ideals. 
One could then expect an interesting interaction 
between the entire class of operator ideals and various algebraic groups of invertible operators, 
with benefits for the study of both of these. 

This paper is devoted to the very first few steps of this program suggested above, 
thereby laying the foundation for a more advanced investigation recorded in 
the follow-up \cite{BPW13}. 
We will first present here a few basic elements of Lie theory 
of linear algebraic groups (Section~\ref{sect2}) 
in order to introduce the structures we are looking for in our infinite-dimensional setting. 
After that, we will introduce the infinite-dimensional algebraic groups  
suitable for our purposes (see Definition~\ref{algebraic} and particularly Definition~\ref{algebraic_restr}) 
and we will establish in Section~\ref{sect3} some very basic facts 
related to the Lie theory for such groups. 
With that notion developed, it makes sense to ask whether the basic features of the structure theory 
of reductive or semisimple Lie algebras can be suitably extended to the classical Lie algebras of operator ideals. 
In this connection we will focus in Section~\ref{sect4} on the study of conjugacy classes of Cartan subalgebras 
(that is, maximal abelian self-adjoint subalgebras). 
That topic is well understood in the case of finite-dimensional semisimple Lie algebras, 
and here we will discuss the phenomena one encounters when one investigates such objects  
for certain infinite-dimensional classical Lie algebras constructed from operator ideals. 
\newpage
\subsection*{General notation convention} 
Throughout this paper we denote by $\Hc$ a separable infinite-dimensional complex Hilbert space, 
by $\Bc(\Hc)$ the unital $C^*$-algebra of all bounded linear operators on $\Hc$ 
with the identity operator $\1\in\Bc(\Hc)$, 
and by $\GL(\Hc)$ the group of all invertible operators in $\Bc(\Hc)$. 
Let $\Kc(\Hc)$ denote the ideal of compact operators on $\Hc$ and  
$\Fc(\Hc)$ denote the ideal of finite-rank operators on $\Hc$. 
By an operator ideal in $\Bc(\Hc)$, 
herein we mean a two-sided ideal of $\Bc(\Hc)$, 
and we say $\Ic$ is a proper ideal if $\{0\}\subsetneqq\Ic\subsetneqq\Bc(\Hc)$, in which case as is well-known,
$\{0\}\subsetneqq\Fc(\Hc)\subseteq\Ic\subseteq\Kc(\Hc)\subsetneqq\Bc(\Hc)$. 
We will use the following notation for the full unitary group on~$\Hc$: 
$$U(\Hc)=\{V\in\Bc(\Hc)\mid V^*V=VV^*=\1\}$$ 
and we will also need its subgroup 
$$U_{\Kc(\Hc)}=U(\Hc)\cap(\1+\Kc(\Hc)) $$
(and denoted by $U_\infty(\Hc)$ for instance in \cite{Ne98}). 
For every $n\ge 1$, we denote by $M_n(\CC)$ the set  
of all $n\times n$ matrices with complex entries, 
endowed with its natural topology when viewed as an $n^2$-dimensional complex vector space. 
Thus, convergence of a sequence in $M_n(\CC)$ means entrywise convergence. 
For all $X,Y\in M_n(\CC)$ we denote $[X,Y]:=XY-YX$ 
and~$X^*$ denotes the conjugate transpose of $X$. 
We also define the general linear group $\GL(n,\CC):=\{T\in M_n(\CC)\mid\det T\ne0\}$, 
i.e., the group of invertible matrices in $M_n(\CC)$, 
which is a dense open subset of~$M_n(\CC)$.  

We will adopt the convention that the Lie groups and the algebraic groups are denoted by 
upper case Latin letters and their Lie algebras are denoted 
by the corresponding lower case Gothic letters. 
In particular, for the unitary group $U(n):=\{V\in M_n(\CC)\mid V^*V=\1\}$, 
which is a compact subgroup of $\GL(n,\CC)$, 
its Lie algebra (see Definition~\ref{Lie} and Remark~\ref{Lie_note}) 
consists of all skew-adjoint matrices and is denoted by $\ug(n):=\{X\in M_n(\CC)\mid X^*=-X\}$.

\section{Linear Lie theory}\label{sect2}

The point of this paper is that we will avoid the notion of Lie group 
and will replace it with the notion of linear algebraic group, 
in order to be able to cover the classical groups associated with operator ideals that do not support any complete norm. 
For that reason, the discussion in this preliminary section is streamlined on two levels. 
Specifically, we will avoid Lie groups in the primary line of development of the discussion (the main definitions made below),  
but we will insert comments on Lie groups in the secondary line of development. 
The reader can easily observe that organization of the discussion, 
since the main definitions in this section can be understood without any knowledge on Lie groups 
but each of these definitions is accompanied by remarks that place it in a proper perspective, 
by looking at from the point of view of Lie groups. 

\subsection*{Basic definitions}

\begin{definition}\label{Lie_alg}
\normalfont
Let $\KK\in\{\RR,\CC\}$. 
A \emph{Lie algebra} over $\KK$ is  
a vector space $\gg$ over $\KK$ endowed with a $\KK$-bilinear map 
(called the Lie bracket) 
$[\cdot,\cdot]\colon\gg\times\gg\to\gg$  
satisfying 
the Jacobi identity 
$$[[X,Y],Z]+[[Y,Z],X]+[[Z,X],Y]=0$$
and the skew-symmetry condition 
$$[X,Y]=-[Y,X]$$ 
for all $X,Y,Z\in\gg$. 
A \emph{subalgebra} of $\gg$ is any $\KK$-linear subspace $\hg\subseteq\gg$ that is closed under the Lie bracket. 
(If moreover $[\gg,\hg]\subseteq\hg$, then $\hg$ is called a \emph{Lie ideal} of $\gg$.) 
The subalgebra $\hg$ is \emph{abelian} if the Lie bracket vanishes on it. 
If $\gg_1$ is another Lie algebra over $\KK$ with the Lie bracket denoted again by $[\cdot,\cdot]$, 
then a $\KK$-linear mapping $\varphi\colon\gg\to\gg_1$ is called a \emph{Lie-homomorphism} 
if for all $X,Y\in\gg$ we have $\phi([X,Y])=[\phi(X),\phi(Y)]$; 
if moreover $\phi$ is a bijection, then we say that it is an \emph{Lie-isomorphism}. 

An \emph{involution} $X\mapsto X^*$ on $\gg$ is 
a mapping  $\gg\to\gg$  
satisfying $(X^*)^*=X$, 
conjugate-linearity $(zX+wY)^*={\bar z}X^*+{\bar w}Y^*$ 
(which is actually linearity in the case $\KK=\RR$), 
and the condition of compatibility of the Lie bracket with the involution, 
$$[X,Y]^*=[Y^*,X^*]$$
for all $z,w\in\KK$ and $X,Y\in\gg$. 

If $\gg$ is a Lie algebra with an involution, then by \emph{Cartan subalgebra} of $\gg$ 
we mean a subalgebra which is maximal in $\gg$ among those abelian subalgebras closed under this involution. 
\end{definition}


\begin{example}\label{assoc}
\normalfont
Again let $\KK\in\{\RR,\CC\}$ and $\Ac$ be an \emph{associative algebra} over $\KK$. 
That is, $\Ac$ is a $\KK$-vector space endowed with a bilinear map 
$\Ac\times\Ac\to\Ac$, $(a,b)\mapsto ab$
which is an \emph{associative product},  
that is, for all $a,b,c\in\Ac$ we have $(ab)c=a(bc)$. 
This associativity condition implies that if we define 
$[a,b]:=ab-ba$ for all $a,b\in\Ac$, 
then we obtain a Lie bracket in the sense of Definition~\ref{Lie_alg}. 
Thus every associative algebra gives rise, in this canonical way, 
to a Lie algebra with the same underlying vector space as the associative algebra itself. 
\end{example}

\begin{example}\label{invol_ex}
\normalfont 
Here are the main examples of involutions on Lie algebras that will be encountered in this paper: 
\begin{enumerate}
\item\label{invol_ex_item1} 
If $\gg$ is a Lie algebra over $\KK=\RR$, called a real Lie algebra, 
then the mapping $X\mapsto -X$ is an involution. 
Thus, all real Lie algebras have involutions.
\item\label{invol_ex_item2} 
If $\gg$ is a $\KK$-linear subspace of $\Bc(\Hc)$ with the property that for all $X,Y\in\gg$ 
it contains the commutator $[X,Y]=XY-YX\in\gg$ and the adjoint operator $X^*\in\gg$, 
then $\gg$ is a Lie algebra over $\KK$ endowed with the involution $X\mapsto X^*$. 
Thus, for instance, $\Bc(\Hc)$ itself is a complex Lie algebra with an involution. 
More generally, any operator ideal in $\Bc(\Hc)$ is a complex Lie algebra with the same involution. 
\end{enumerate}
Note that the set of all skew-adjoint operators 
$$\ug(\Hc)=\{X\in\Bc(\Hc)\mid X^*=-X\}$$
is a real Lie algebra for which the aforementioned two involutions 
(\eqref{invol_ex_item1} and \eqref{invol_ex_item2}) coincide. 
More generally, if a real linear subspace $\gg\subseteq\ug(\Hc)$ is closed under operator commutators, 
then $\gg$ is a real Lie algebra for which the above two involutions coincide. 
\end{example}

\begin{remark}\label{PBW}
\normalfont
The above example also shows that the Lie bracket of a Lie algebra may not 
come from an associative product on that algebra as in Example~\ref{assoc}, 
in that, if $X,Y\in\ug(\Hc)$ 
then $[X,Y]=XY-YX\in\ug(\Hc)$ although it may happen that $XY\not\in\ug(\Hc)$. 
Examples of such situations are easily constructed: if $0\ne X=-X^*\in\Bc(\Hc)$ 
and $Y:=X$, then $X,Y\in\ug(\Hc)$ and $XY=X^2\ne 0$, hence $(XY)^*=X^2=XY\ne-XY$, 
and thus $XY\not\in\ug(\Hc)$.

Concerning the general relationship between Lie algebras and 
associative algebras, we recall from Example~\ref{assoc} 
that every associative algebra can be regarded 
as a Lie algebra with the Lie bracket defined as the single commutator of its elements. 
In the converse direction (that is, from Lie algebras towards associative algebras), 
it follows by the Poincar\'e-Birkhoff-Witt theorem 
that for every Lie algebra $\gg$ over $\KK$ 
there exists an associative algebra $\Ac$ over $\KK$ such that 
$\gg$ is a \emph{Lie subalgebra} of $\Ac$, in the sense that $\gg$ is a $\KK$-linear subspace 
of $\Ac$ and for all $X,Y\in\gg$ we have $[X,Y]=XY-YX$, 
where the left-hand side uses the Lie bracket of $\gg$ while the right-hand side 
involves the products computed in $\Ac$. 
We recall that such an algebra $\Ac$ is constructed 
as the quotient of the tensor algebra $\bigoplus\limits_{k\ge0}(\otimes^k \gg)$ by 
its two-sided ideal generated by the subset $\{X\otimes Y-Y\otimes X-[X,Y]\mid X,Y\in\gg\}\subseteq(\gg\otimes\gg)\oplus\gg$; 
see for instance \cite[Chapter III]{Kn96} or \cite[Appendix C]{GW09}. 
See also Remark~\ref{Ado} below for additional information on the relationship between 
Lie algebras and associative algebras in the finite-dimensional case. 
\end{remark}

The above setting of Lie algebras will now be specialized to finite dimensions until the 
end of this section. 
This will then serve as motivation for the following sections 
in our study of infinite-dimensional structures (particularly algebraic groups).
We will begin this discussion by 
introducing a special class of Lie algebras with involutions 
in the sense of Definition~\ref{Lie_alg}.

\begin{definition}\label{red_alg}
\normalfont 
A real (respectively, complex) \emph{linear Lie algebra} is a real (respectively, complex) 
linear subspace $\gg\subseteq M_n(\CC)$ for some integer $n\ge1$
such that for all $X,Y\in\gg$ we have $[X,Y]\in\gg$. 
Moreover, we say that $\gg$ is a \emph{linear reductive Lie algebra} 
if for every $X\in\gg$ we have $X^*\in\gg$, 
where we recall that $X^*$ denotes the conjugate transpose of~$X$. 
\end{definition}

\begin{remark}\label{Ado}
\normalfont
Every finite-dimensional Lie algebra is isomorphic to a linear Lie algebra in the sense of Definition~\ref{red_alg}, 
as a consequence of Ado's theorem; see for instance \cite[2.5.5--6]{Di77}.  
Nevertheless, the terminology of Definition~\ref{red_alg} is useful as a Lie algebraic counterpart of the notion introduced in Definition~\ref{red_gr} below, 
in order to emphasize that the various versions of Lie groups are studied by using their Lie algebras. 

Also, the notion of a linear Lie algebra provides a way of defining reductive Lie algebras, 
which fits well with the infinite-dimensional Lie algebras constructed out of arbitrary operator ideals 
in the later sections of this paper.  
More traditional is the equivalent way of defining a reductive Lie algebra 
as a finite-dimensional Lie algebra $\gg$  with the property that 
to each ideal $\ag$ in $\gg$ there corresponds an ideal $\bg$ in $\gg$ 
such that we have the direct sum decomposition $\gg=\ag\oplus\bg$ (see for instance \cite{Kn96}).  
This approach can be directly extended to the classical Lie algebras 
constructed out of the Hilbert-Schmidt ideal 
and was pursued in \cite{Sch60} and \cite{Sch61}.  
However, it does not seem to be easily adapted to arbitrary operator ideals. 
\end{remark}

\begin{definition}\label{Lie}
\normalfont
If $G$ is a closed subgroup of $\GL(n,\CC)$, then 
the \emph{Lie algebra of $G$} is 
$\gg:=\{X\in M_n(\CC)\mid(\forall t\in\RR)\quad \exp(tX)\in G\}$. 
\end{definition}

\begin{remark}\label{Lie_note}
\normalfont
We recall that $\GL(n,\CC)$ is a Lie group and every closed subgroup of a Lie group 
is in turn a Lie group;  
see for instance \cite{Kn96}. 
In particular, the group $G$ from Definition~\ref{Lie} has the natural structure of a Lie group, 
and its Lie algebra defined above agrees with 
the notion of the Lie algebra defined by using the differentiable structure as in  
 \cite[Chapter I, Section 10]{Kn96}. 
The description of Lie algebras from Definition~\ref{Lie} (also used for instance in \cite[Subsection 1.3.1]{GW09}) 
has the advantage that it allows for the concrete computation of the Lie algebra under consideration, 
and moreover it extends directly to groups 
related to general operator ideals; see Theorem~\ref{alg_indeed}.
\end{remark}

We now recall the linear reductive groups after \cite[Def. 2.5]{Vo00}. 

\begin{definition}\label{red_gr}
\normalfont
A \emph{linear reductive group} is a closed subgroup $G$ of 
$\GL(n,\CC)$ 
for some $n\ge1$ satisfying the following conditions: 
\begin{enumerate}
\item for every $T\in G$ we have $T^*\in G$; 
\item $G$ has finitely many connected components. 
\end{enumerate}
\end{definition}

We will define the linear algebraic groups below as in \cite[Def. 1.4.1]{GW09}. 
By using the natural embeddings 
$$\GL(n,\CC)\hookrightarrow\GL(2n,\RR)\hookrightarrow\GL(2n,\CC),$$ 
the real linear algebraic groups in the sense of the following definition 
can be regarded as the groups of $\RR$-rational points of suitable linear algebraic groups in $\GL(2n,\CC)$, 
in the terminology of \cite[Def. 1.7.1]{GW09} and \cite[2.1.1]{Sp98}. 
We prefer the terminology to be introduced below since it is closer related 
to the one already used in \cite{HK77} for the algebraic groups in infinite dimensions  
and moreover it was also used even for finite-dimensional algebraic groups for instance in \cite{Bor01}. 
The following definition will be extended in Definition~\ref{algebraic} below. 
Recall that  
the entries of the inverse of a matrix $T\in\GL(n,\CC)$ are given by certain 
quotients of polynomials in terms of the entries of $T$, 
the coefficients of these polynomials depending only on $n\ge1$ 
and on the position of the corresponding entry of~$T^{-1}$. 
Therefore it is easily seen that if a subgroup of $\GL(n,\CC)$ is equal to 
the set of solutions to some polynomial equations depending on $T$ and $T^{-1}$ 
(as in Definition~\ref{algebraic}), 
then it is also given by a set of polynomial equations depending only on $T$, 
as in the following definition. 

\begin{definition}\label{alg_gr}
\normalfont
A \emph{real linear algebraic group} is a subgroup $G$ of 
$\GL(n,\CC)$ 
for some $n\ge1$ such that 
there exists a family $\Pc$ of not necessarily holomorphic polynomials on $M_n(\CC)$ 
with  
$$G=\{T\in \GL(n,\CC)\mid (\forall p\in\Pc)\quad p(T)=0\}.$$ 
By a not necessarily holomorphic polynomial on $M_n(\CC)$ we mean any complex valued 
function on $M_n(\CC)$ defined by a polynomial in the matrix entries and their complex conjugates.

If we also have $T^*\in G$ for every $T\in G$, then $G$ will be called 
a \emph{real reductive linear algebraic group}. 
On the other hand, if the above set $\Pc$ can be chosen to consist only of 
holomorphic polynomials 
(i.e., involving no complex conjugates of matrix entries), 
then we say that $G$ is a \emph{linear algebraic group}  
or a \emph{reductive linear algebraic group}, respectively. 
\end{definition}

The following definition describes precisely the compact Lie groups (see \cite{Kn96} 
and Remark~\ref{trick} below) 
and we state it in this way since it fits well with the purpose of the present paper. 
Namely, it emphasizes the existence of a particular realization of a compact Lie group, 
rather than its differentiable structure. 

\begin{definition}\label{comp_gr}
\normalfont
A \emph{compact linear group} is a closed subgroup of the unitary group $U(n)$ for some integer $n\ge1$. 
\end{definition}

\begin{remark}\label{trick}
\normalfont
For further motivation of the above terminology, 
recall that $U(n)$ is a compact group hence its closed subgroups are in turn compact. 
Conversely, every compact subgroup of $\GL(n,\CC)$ is a compact linear group 
in the sense of Definition~\ref{comp_gr} after a suitable change of basis in $\CC^n$, 
where $\GL(n,\CC)$ is regarded as the group of all invertible linear operators on $\CC^n$. 
More specifically, one can use the so-called Weyl's unitarian trick (\cite[Prop. 4.6]{Kn96}) 
for defining a scalar product on $\CC^n$ 
that is invariant under the action of every operator in the compact group~$G$, 
and thus $G$ can be viewed as a group of unitary matrices by using a basis in $\CC^n$ 
which is orthonormal with respect to the new scalar product.
\end{remark}

In connection with the above notions, we note that the implications 
$$\begin{matrix}
\text{linear reductive group (Definition~\ref{red_gr})}\\ 
\Uparrow \\
\text{real reductive linear algebraic group (Definition~\ref{alg_gr})}\\ 
\Uparrow \\
\text{compact linear group (Definition~\ref{comp_gr})} 
\end{matrix}$$
hold true; see for instance \cite[Subsection 5.2]{Bor01} for the bottom implication, 
while the implication from the top is obvious from the definitions. 

\begin{remark}\label{red_facts}
\normalfont  We record a few simple facts related to the above definitions. 
\begin{enumerate}
\item\label{red_facts_item1} 
Every linear reductive Lie algebra is a Lie algebra with an involution in the sense of Definition~\ref{Lie_alg}. 
\item\label{red_facts_item2} 
The Lie algebra of a closed subgroup $G$ of $\GL(n,\CC)$ is a linear Lie algebra 
in the sense of Definition~\ref{red_alg}. 
See Remark~\ref{closed_alg} for a proof of this fact in a more general setting. 
\item\label{red_facts_item4} 
If $G$ is a closed group of  $\GL(n,\CC)$ with the Lie algebra $\gg$, 
then for all $T\in G$ and $Y\in\gg$ we have $TYT^{-1}\in\gg$ 
and the mapping 
$$\Ad_G\colon G\times\gg\to\gg,\quad (T,Y)\mapsto\Ad_G(T)Y:=TYT^{-1}$$
is called the \emph{adjoint action} of $G$. 
Moreover, one has the mapping 
$$\ad_{\gg}\colon\gg\times\gg\to\gg,\quad (X,Y)\mapsto(\ad_{\gg}X)Y:=[X,Y]=XY-YX$$
called the \emph{adjoint representation} of~$\gg$. 

The notation convention for $\Ad$ and $\ad$  is related to the one of denoting the linear algebraic groups 
or the Lie groups by upper case (Roman) letters
and the Lie algebras by lower case (Gothic) leters.
The connection between  $\Ad$ and $\ad$
is that if one picks a 1-parameter group $t\mapsto\exp(tX)$ with $X\in\gg$ 
and one differentiates the corresponding 1-parameter group of linear transformations 
$\Ad_G(\exp(tX))\colon\gg\to\gg$  at $t=0$, 
then one obtains $\ad_{\gg} X$, and geometrically this reflects the fact that the Lie algebra of $G$ 
is the tangent space at $\1\in G$.
For instance the unit circle can be viewed as a compact linear group and its Lie algebra is the tangent line at~1,
viewed as a vector space with the origin at that point~1.
\end{enumerate}
\end{remark}

\subsection*{Representations of linear reductive Lie groups}
The main problem with the above definitions of linear reductive groups and their Lie algebras is 
that they depend on the embeddings of these objects into a matrix algebra. 
For instance, if $G\subseteq M_n(\CC)$ is a linear reductive group, 
then for every $l\ge 1$ one gets another embedding 
$G\hookrightarrow M_{n+l}(\CC)$, $T\mapsto \begin{pmatrix} T & 0 \\ 0 & \1\end{pmatrix}$, 
and so on. 
It is then natural to wonder about the embeddings which are minimal in some reasonable sense. 
In connection with this problem, we will focus on the embeddings $G\subseteq M_n(\CC)$ which are irreducible, 
that is, the commutant of $G$ is precisely all scalar multiplies of the identity.

For the sake of simplicity we will discuss only the case when $G$ is a compact linear group, 
and we will consider mappings that are slightly more general than the embeddings, namely the representations, 
in the sense of the following definition. 

\begin{definition}
\normalfont
If $G$ is a compact linear group and $m\ge 1$ is an integer, then a \emph{unitary $m$-dimensional representation} of $G$ 
is a continuous mapping $\pi\colon G\to U(m)$ such that for all $T,S\in G$ we have $\pi(TS)=\pi(T)\pi(S)$. 
We say $\pi$ is a \emph{unitary irreducible} representation if 
the scalar multiples of $\1\in M_m(\CC)$ are the only matrices in $M_m(\CC)$ that commute with every matrix 
$\pi(T)$ with $T\in G$. 
\end{definition}

\begin{remark}\label{schol}
\normalfont
The classification of unitary irreducible representations 
of a compact linear group relies quite heavily on 
the Cartan subalgebras of the Lie algebra $\gg$ of $G$. 
In order to explain that point and to motivate the problems 
we will address for operator Lie algebras and algebraic groups in the later sections of the present paper, 
we sketch below the method of work that eventually leads to 
the aforementioned classification. 

\emph{Step 1}. 
If $\pi\colon G\to U(m)$ is a unitary irreducible representation,  
then one can define its differential by 
$$\de\pi\colon\gg\to\ug(m),\quad \de\pi(X)=\frac{\de}{\de t}\Big\vert_{t=0}\pi(\exp(tX))$$
and then one can prove that $\de\pi$ is a homomorphism of Lie algebras, hence 
\begin{equation}\label{schol_eq1}
(\forall X,Y\in\gg)\quad \de\pi([X,Y])=[\de\pi(X),\de\pi(Y)].
\end{equation}

\emph{Step 2}. 
If we pick a Cartan subalgebra $\hg\subset\gg$, then \eqref{schol_eq1} implies that 
$\de\pi(\hg)$ is a linear subspace of $M_m(\CC)$ consisting of \emph{mutually commuting} skew-adjoint matrices. 
Hence the Hilbert space $\CC^m$ splits into the orthogonal direct sum 
$$\CC^m=\bigoplus_{j=1}^k\Vc_j,$$
where there exist distinct linear functionals $\lambda_1,\dots,\lambda_k\colon\hg\to\RR$, 
to be called the \emph{weights} of the representation $\pi$, 
such that 
$$\Vc_j=\{v\in\CC^m\mid(\forall X\in\hg)\quad \de\pi(X)v=i\lambda_j(X)v\}
\text{ for }j=1,\dots,k.$$

\emph{Step 3}. 
If we now pick a basis $H_1,\dots,H_{\ell}$ in $\hg$, 
then we obtain a linear isomorphism from the dual linear space of $\hg$ onto $\RR^{\ell}$ 
by 
$$\hg^*\to\RR^{\ell},\quad \lambda\mapsto(\lambda(H_1),\dots,\lambda(H_{\ell}))$$
By using this isomorphism we can transport the lexicographic ordering from $\RR^{\ell}$ to~$\hg^*$. 
We thus get a total ordering on $\hg^*$ and after a renumbering of the weights of the representation $\pi$ 
we may assume that $\lambda_1\ge\cdots\ge\lambda_{\ell}$. 

\emph{Step 4}. 
The theorem of the highest weight (see \cite{Kn96}) roughly says that, 
after a Cartan subalgebra $\hg$ and a basis in that Cartan subalgebra have been fixed as above, 
each unitary irreducible representation is uniquely determined (up to a unitary equivalence) 
by its highest weight. 
The theorem also characterizes the linear functionals on $\hg$ 
that can occur as highest weights of unitary irreducible representations. 
\end{remark}

In connection with the above discussion, 
we recall that the classification of the unitary irreducible representations 
does not really depend on the choice of the Cartan subalgebra, 
and this important fact basically follows from the following
conjugation theorem for Cartan subalgebras: 

\begin{theorem}\label{first_th}
If $G$ is a compact linear group whose Lie algebra $\gg$ 
is endowed with the involution $X\mapsto -X$,     
then any two Cartan subalgebras $\hg_1$ and $\hg_2$ of $\gg$ are 
\emph{$G$-conjugated} to each other. 
That is, there exists 
$g\in G$ such that $\Ad_G(g)\hg_1=\hg_2$.  
\end{theorem}

\begin{proof}
See for instance \cite[Th. 4.34]{Kn96} and recall that the compact linear groups 
are precisely the compact Lie groups.
\end{proof}

\begin{example}\label{Un}
\normalfont 
We will illustrate the above discussion by the compact Lie group 
$$G=U(n)=\{V\in M_n(\CC)\mid V^*V=\1\},$$
whose Lie algebra is
$$\gg=\ug(n)=\{X\in M_n(\CC)\mid X^*=-X\}.$$
There is a $U(n)$-equivariant one-to-one correspondence between the Cartan subalgebras of $\gg$ 
and the \emph{complete flags} 
in $\CC^n$, that is, increasing families of linear subspaces 
$$\{0\}\subsetneqq\Vc_1\subsetneqq\cdots\subsetneqq\Vc_{n-1}\subsetneqq\Vc_n=\CC^n.$$
For this reason, the set of all Cartan subalgebras of $\gg$ 
(as well as of other Lie algebras) is called the `flag manifold' or `flag variety'. 
See for instance \cite{Wo98} and \cite{Vo08} for the differential geometry of 
the various flag manifolds and their role in the representation theory of reductive Lie groups. 

In the special case under consideration, the fact that any two maximal abelian subalgebras of $\gg$  
(i.e., Cartan subalgebras of $\gg$)
are mapped to each other by the unitary equivalence $X\mapsto VXV^*$ 
for a suitable $V\in U(n)$, is equivalent to the spectral theorem for skew-adjoint matrices 
(compare the remark after the statement of \cite[Th. 4.36]{Kn96}). 
More specifically, note that the set $\hg_0$ of all skew-adjoint diagonal matrices in $M_n(\CC)$ 
is a particular Cartan subalgebra of $\gg$.  
On the other hand, every element $X\in\gg=\ug(n)$ belongs to some maximal abelian subalgebra, say $\hg_1$, 
and then  $VXV^{-1}\in V\hg_1V^{-1}=\hg_0$ 
for some $V\in U(n)$ by Theorem~\ref{first_th}.  
That is, $VXV^{-1}$ is a diagonal matrix and we have thus obtained the spectral theorem for the skew-adjoint matrix~$X$. 
Conversely, if $\hg_0$ is as above and $\hg_1$ is any maximal abelian subalgebra of $\gg$, 
then there exists $X_1\in\hg_1$ for which we have $\hg_1=\{Y\in\gg\mid[X_1,Y]=0\}$ 
(see \cite[Lemma 4.33]{Kn96}). 
It follows by the spectral theorem for $X_1\in\hg_1\subseteq\gg=\ug(n)$ that there exists 
$V\in U(n)$ such that $VX_1V^{-1}\in\hg_0$, and this easily implies $V\hg_1V^{-1}\subseteq\hg_0$ 
because of the way $X_1$ was chosen 
(compare the proof of \cite[Th. 4.34]{Kn96}). 
Since both $\hg_1$ and $\hg_0$ are maximal abelian subalgebras of $\gg$, it then follows that $V\hg_1V^{-1}=\hg_0$, 
hence the Cartan subalgebras $\hg_1$ and $\hg_0$ are $U(n)$-conjugated to each other. 
\end{example}

\newpage

\section{Lie theory for some infinite-dimensional algebraic groups}\label{sect3}

\subsection*{Infinite-dimensional linear algebraic reductive groups} 
The notion of linear algebraic group in infinite dimensions requires 
the following terminology. 
If $\Ac$ is a real Banach space, 
then a vector-valued continuous polynomial function on $\Ac$ of degree $\le n$
is a function $p\colon\Ac\to\Vc$, where $\Vc$ is another real Banach space, 
such that for some continuous $k$-linear maps 
$$\psi_k\colon
\underbrace{\Ac\times\cdots\times\Ac}_{\text{$k$ times}}
\to\Vc$$
(for $k=0,1,\ldots,n$) we have
$p(a)=\psi_n(a,\ldots,a)+\cdots+\psi_1(a)+\psi_0$ 
for every $a\in\Ac$, where $\psi_0\in\Vc$. 

Now let $\Bg$ be a real \emph{associative} unital Banach algebra, 
hence a real Banach space endowed with a bounded bilinear mapping $\Bg\times\Bg\to\Bg$, $(x,y)\mapsto xy$  
which is associative and admits a unit element $\1\in\Bg$.
Then the set 
$$\Bg^\times:=\{x\in\Bg\mid(\exists y\in\Bg)\ xy=yx=\1\}$$
is an open subset of $\Bg$ and 
has the natural structure of a Banach-Lie group (\cite[Example 6.9]{Up85}). 
The Lie algebra of $\Bg^\times$ is 
again the Banach space $\Bg$, viewed however as a \emph{nonassociative} Banach algebra, 
more precisely as a Banach-Lie algebra whose Lie bracket is the bounded bilinear mapping 
$\Bg\times\Bg\to\Bg$, $(x,y)\mapsto xy-yx$. 

\begin{definition}\label{closed}
\normalfont
If $\Bg$ is a real associative unital Banach algebra and 
$G$ is a closed subgroup of $\Bg^\times$, then 
the \emph{Lie algebra of $G$} is 
$$\gg:=\{x\in\Bg\mid(\forall t\in\RR)\quad\exp(tx)\in G\}.$$
\end{definition}

\begin{remark}\label{closed_alg}
\normalfont
In the setting of Definition~\ref{closed}, 
the set $\gg$ is a closed Lie subalgebra of $\Bg$ 
(\cite[Corollary 6.8]{Up85}). 

In fact, since $G$ is a closed subset of $\Bg^\times$, it is easily seen that $\gg$ is closed in $\Bg$. 
Moreover, by using the well-known formulas (\cite[Proposition 6.7]{Up85})
$$\begin{aligned}
\exp(t(x+y))&=\lim_{k\to\infty}\Bigl(\exp(\frac{t}{k}x)\exp(\frac{t}{k}y)\Bigr)^k, \\
\exp(t^2[x,y])&=\lim_{k\to\infty}\Bigl(\exp(\frac{t}{k}x)\exp(\frac{t}{k}y)\exp(-\frac{t}{k}x)\exp(-\frac{t}{k}y)\Bigr)^{k^2}
\end{aligned}$$
which hold true for all $x,y\in\Bg$ and $t\in\RR$, 
it follows that for every $x,y\in\gg$ we have $x+y\in\gg$ and $[x,y]\in\gg$. 
Then it is easy to check that $\gg$ is a linear subspace of $\Bg$. 

Moreover, if $\Bg$ is endowed with a continuous involution such that for every $b\in G$ we have $b^*\in G$, 
then for every $x\in\gg$ we have $x^*\in\gg$.
\end{remark}

\begin{definition}[\cite{HK77}]\label{algebraic}
\normalfont
Let $\Bg$ be a real 
associative unital Banach algebra, $n$ be a positive integer, 
and $G$ be  a subgroup of $\Bg^\times$. 
We say that $G$ is an 
\emph{algebraic group in $\Bg$ of degree $\le n$}
if we have 
$$G=\{b\in \Bg^\times\mid (\forall p\in\Pc)\quad p(b,b^{-1})=0\}$$
for some set $\Pc$ of vector-valued continuous polynomial functions on $\Bg\times\Bg$. 
Note that $G$ is a closed subgroup of $\Bg^\times$, hence its Lie algebra can be defined 
as in Definition~\ref{closed}. 

If moreover $\Bg$ is endowed with a continuous involution $b\mapsto b^*$ 
and for every $b\in G$ we have $b^*\in G$, then we say that the group $G$ is \emph{reductive}. 
\end{definition}

\begin{definition}\label{algebraic_restr}
\normalfont
Let $\Bg$ be a real 
associative unital Banach algebra, $n$ be a positive integer, 
and $G$ be an algebraic subgroup of $\Bg^\times$ of degree $\le n$ 
with the Lie algebra $\gg$ ($\subseteq\Bg$). 
Then for every one-sided ideal $\Ig$ of $\Bg$,
the corresponding 
\emph{$\Ig$-restricted} algebraic group is  
$$G_{\Ig}:=G\cap(\1+\Ig)$$
and the \emph{Lie algebra of $G$} is $\gg_{\Ig}:=\gg\cap\Ig$.  
%
\end{definition}

\begin{remark}\label{simple}
\normalfont 
Here are some simple remarks
 on algebraic structures that occur in the preceding definition.  
Let $\Bg$ be a unital ring and $\Ig$ be a one-sided ideal of $\Bg$. 
 \begin{enumerate}
\item\label{simple_item1} The set
$\Bg^\times\cap(\1+\Ig)$ is always a subgroup of the group $\Bg^\times$ 
of invertible elements in $\Bg$. 

To see this, let us assume for instance that $\Ig\Bg\subseteq\Ig$. 
Then $\Ig\Ig\subseteq\Ig$, hence $(\1+\Ig)(\1+\Ig)\subseteq\1+\Ig$, 
and thus $(\1+\Ig)\cap\Bg^\times$ is closed under the product. 
On the other hand, if $x\in\Ig$, $b\in\Bg$ and $(\1+x)b=\1$, then $x=\1-xb\in\1+\Ig\Bg\subseteq\1+\Ig$, 
hence $(\1+\Ig)\cap\Bg^\times$ is also closed under the inversion. 

\item\label{simple_item2} 
By definition, every one-sided ideal of a real algebra is assumed to be a real linear subspace. 
Therefore, if the unital ring $\Bg$ has the structure of a real algebra, 
then $\Ig$ is an associative subalgebra of $\Bg$ and in particular $\Ig$ 
has the natural structure of a real Lie algebra with the Lie bracket 
defined by $[x,y]:=xy-yx$ for all $x,y\in\Ig$. 

\item\label{simple_item3} If $\Bg$ is a ring endowed with an involution $b\mapsto b^*$ and $\Ig$ is 
a self-adjoint one-sided ideal of $\Bg$, then $\Ig$ is actually a two-sided ideal. 

In fact, if we assume for instance $\Ig\Bg\subseteq\Ig$, 
then for every $x\in\Ig$ and $b\in\Bg$ we have $x^*b^*\in\Ig$ hence $bx=(x^*b^*)^*\in\Ig$, 
and thus $\Bg\Ig\subseteq\Ig$ as well. 
\end{enumerate}
\end{remark}

%
%

\begin{theorem}\label{alg_indeed}
Let $\Bg$ be a real associative unital Banach algebra with a one-sided ideal $\Ig$. 
If $G_{\Ig}$ is an $\Ig$-restricted algebraic group in $\Bg$, 
then its Lie algebra is a Lie subalgebra of $\Ig$ and can be described as 
$$\gg_{\Ig}=\{x\in \Bg\mid(\forall t\in\RR)\quad \exp(tx)\in G_{\Ig}\}.$$ 
\end{theorem}

\begin{proof}
See \cite{BPW13}. 
\end{proof}


\subsection*{Classical groups and Lie algebras in infinite dimensions}

We will now provide several examples of linear algebraic reductive groups 
associated with operator ideals, by way of illustrating Definition~\ref{algebraic_restr}. 
To this end we elaborate on an idea from \cite[Probl. 3.4]{Be09} 
by introducing the classical groups and 
Lie algebras associated to an arbitrary operator ideal.  
The ones associated with 
the Schatten ideals $\Sg_p(\Hc)$ ($1\le p\le\infty$) 
were first systematically studied in \cite{dlH72}.

\begin{definition}\label{complex_gr}
\normalfont
Let $\Ig$ be an arbitrary ideal in $\Bc(\Hc)$. 
We define the following groups and complex Lie algebras:  
\begin{itemize}
\item[{\rm(A)}] the \emph{complex general linear group} 
$$\GL_{\Ig}(\Hc)=\GL(\Hc)\cap(\1+\Ig)$$
with the Lie algebra 
$$
\gl_{\Ig}(\Hc):=\Ig;$$ 
\item[{\rm(B)}] the \emph{complex orthogonal group} 
$$\OO_{\Ig}(\Hc):=\{T\in\GL_{\Ig}(\Hc)\mid T^{-1}=JT^*J^{-1}\}$$
with the Lie algebra 
$$
\og_{\Ig}(\Hc):=\{X\in\Ig\mid X=-JX^*J^{-1}\},$$ 
where $J\colon\Hc\to\Hc$ is a conjugation 
(i.e., $J$ a conjugate-linear isometry satisfying $J^2=\1$); 
\item[{\rm(C)}] the \emph{complex symplectic group} 
$$\Sp_{\Ig}(\Hc):=\{T\in\GL_{\Ig}(\Hc)\mid 
 T^{-1}=\widetilde{J}T^*\widetilde{J}^{-1}\}$$ 
with the Lie algebra 
$$
\sp_{\Ig}(\Hc):=\{X\in\Ig\mid X=-\widetilde{J}X^*\widetilde{J}^{-1}\},$$ 
where $\widetilde{J}\colon\Hc\to\Hc$ is an anti-conjugation 
(i.e., $\widetilde{J}$ a conjugate-linear isometry satisfying 
$\widetilde{J}^2=-\1$).
\end{itemize}
We shall say that $\GL_{\Ig}(\Hc)$, $\OO_{\Ig}(\Hc)$, and $\Sp_{\Ig}(\Hc)$ 
are the \textit{classical complex groups associated with  
the operator ideal $\Ig$}. 
Similarly, the corresponding Lie algebras 
we call the \textit{classical complex Lie algebras} 
(associated with $\Ig$). 
%
\end{definition}

\begin{definition}\label{real_gr}
\normalfont
We shall use the notation of Definition~\ref{complex_gr} 
and define the following groups and real Lie algebras 
associated to the operator ideal~$\Ig$:  
\begin{itemize}
\item[{\rm(AI)}] the \emph{real general linear group} 
$$\GL_{\Ig}(\Hc;{\mathbb R})=\{T\in\GL_{\Ig}(\Hc)\mid TJ=JT\}$$ 
with the Lie algebra 
$$
\gl_{\Ig}(\Hc;{\mathbb R})
:=\{X\in\Ig\mid XJ=JX\},$$ 
where $J\colon\Hc\to\Hc$ is any conjugation on $\Hc$; 
\item[{\rm(AII)}] the \emph{quaternionic general linear group} 
$$\GL_{\Ig}(\Hc;{\mathbb H})=\{T\in\GL_{\Ig}(\Hc)\mid 
 T\widetilde{J}=\widetilde{J}T\}$$ 
with the Lie algebra 
$$
\gl_{\Ig}(\Hc;{\mathbb H})
:=\{X\in\Ig\mid X\widetilde{J}=\widetilde{J}X\},$$ 
where $\widetilde{J}\colon\Hc\to\Hc$ is any anti-conjugation on $\Hc$, 
which defines on $\Hc$ the structure of a vector space over the quaternion field ${\mathbb H}$  
such that the operators in $\GL_{\Ig}(\Hc;{\mathbb H})$ and $\gl_{\Ig}(\Hc;{\mathbb H})$ are ${\mathbb H}$-linear;
\item[{\rm(AIII)}] the \emph{pseudo-unitary group} 
$$\U_{\Ig}(\Hc_{+},\Hc_{-}):=\{T\in\GL_{\Ig}(\Hc)\mid T^*VT=V\}$$
with the Lie algebra 
$$
\ug_{\Ig}(\Hc_{+},\Hc_{-})
:=\{X\in\Jg\mid X^*V=-VX\},$$
where $\Hc=\Hc_{+}\oplus\Hc_{-}$ and 
$V=\begin{pmatrix} \hfill 1 & \hfill 0 \\ \hfill 0 & \hfill -1\end{pmatrix}$ 
with respect to 
this orthogonal direct sum decomposition of $\Hc$; 
\item[{\rm(BI)}] the \emph{pseudo-orthogonal group} 
$$\OO_{\Ig}(\Hc_{+},\Hc_{-}):=\{T\in\GL_{\Ig}(\Hc)\mid T^{-1}=JT^*J^{-1} 
\text{ and }g^*Vg=V\}$$
with the Lie algebra 
$$
\og_{\Ig}(\Hc_{+},\Hc_{-})
:=\{X\in\Jg\mid X=-JX^*J^{-1}\text{ and }X^*V=-VX\},$$
where $\Hc=\Hc_{+}\oplus\Hc_{-}$, 
$V=\begin{pmatrix} \hfill 1 & \hfill 0 \\ \hfill 0 & \hfill -1\end{pmatrix}$ 
with respect to 
this orthogonal direct sum decomposition of $\Hc$, 
and $J\colon\Hc\to\Hc$ is a conjugation on $\Hc$ 
such that $J(\Hc_{\pm})\subseteq\Hc_{\pm}$; 
\item[{\rm(BII)}]
$\OO^*_{\Ig}(\Hc):=\{T\in\GL_{\Ig}(\Hc)\mid T^{-1}=JT^*J^{-1} 
\text{ and }g\widetilde{J}=\widetilde{J}g\}$ 
with the Lie algebra 
$$
\og^*_{\Ig}(\Hc)
:=\{X\in\Jg\mid X=-JX^*J^{-1}\text{ and }X\widetilde{J}=\widetilde{J}X\},$$
where  $J\colon\Hc\to\Hc$ is a conjugation and 
$\widetilde{J}\colon\Hc\to\Hc$ is an anti-conjugation
such that $J\widetilde{J}=\widetilde{J}J$; 
\item[{\rm(CI)}] 
$\Sp_{\Ig}(\Hc;{\mathbb R}):=\{T\in\GL_{\Ig}(\Hc)\mid 
T^{-1}=\widetilde{J}T^*\widetilde{J}^{-1}\text{ and }TJ=JT\}$ 
with the Lie algebra 
$$\sp_{\Ig}(\Hc;{\mathbb R}):=\{X\in\Ig\mid 
-X=\widetilde{J}X^*\widetilde{J}^{-1}\text{ and }XJ=JX\},
$$
where $\widetilde{J}\colon\Hc\to\Hc$ is any anti-conjugation and  
$J\colon\Hc\to\Hc$ is any conjugation such that 
$J\widetilde{J}=\widetilde{J}J$; 
\item[{\rm(CII)}] 
$\Sp_{\Ig}(\Hc_{+},\Hc_{-}):=\{T\in\GL_{\Ig}(\Hc)\mid 
T^{-1}=\widetilde{J}T^*\widetilde{J}^{-1}\text{ and }T^*VT=V\}$ 
with the Lie algebra 
$$
\sp_{\Ig}(\Hc_{+},\Hc_{-})
:=\{X\in\Jg\mid X=-\widetilde{J}X^*\widetilde{J}^{-1}\text{ and }X^*V=-VX\},$$
where $\Hc=\Hc_{+}\oplus\Hc_{-}$, 
$V=\begin{pmatrix} \hfill 1 & \hfill 0 \\ \hfill 0 & \hfill -1\end{pmatrix}$ 
with respect to 
this orthogonal direct sum decomposition of $\Hc$, 
and $\widetilde{J}\colon\Hc\to\Hc$ is an anti-conjugation on $\Hc$ 
such that $\widetilde{J}(\Hc_{\pm})\subseteq\Hc_{\pm}$.
\end{itemize}\quad \\

\noindent We say  
$\GL_{\Ig}(\Hc;{\mathbb R})$, $\GL_{\Ig}(\Hc;{\mathbb H})$, 
$\U_{\Ig}(\Hc_{+},\Hc_{-})$, 
$\OO_{\Ig}(\Hc_{+},\Hc_{-})$, $\OO^*_{\Ig}(\Hc)$, 
$\Sp_{\Ig}(\Hc;{\mathbb R})$, and $\Sp_{\Ig}(\Hc_{+},\Hc_{-})$
are the \textit{classical real groups associated with
the operator ideal $\Ig$}. 
Similarly, the corresponding Lie algebras 
are called the \textit{classical real Lie algebras} 
(associated with $\Ig$). 

If any of the subspaces $\Hc_{+}$ or $\Hc_{-}$ is equal to $\{0\}$ 
then we will omit it from the notation of any of the above groups and Lie algebras 
of type (AIII), (BI), and (CII). 
For instance, if $\Hc_{-}=\{0\}$ (hence $\Hc_{-}=\Hc$ and $V=\1$), 
then we will write 
$$U_{\Ic}(\Hc):=U(\Hc)\cap(\1+\Ic)$$
and so on. 
\end{definition}

\begin{remark}\label{class_dlH}
\normalfont
As a by-product of 
the classification of the $L^*$-algebras 
(see for instance Theorems 7.18~and~7.19 in \cite{Be06}), 
every (real or complex) topologically simple $L^*$-algebra 
is isomorphic to one of
the classical Banach-Lie algebras associated with 
the Hilbert-Schmidt ideal $\Jg=\Sg_2(\Hc)$.
\end{remark}

\begin{example}\label{ideals_ex}
\normalfont
We will now illustrate the wide variety of examples that prompted us to search for an alternative to the study of operator ideals in the framework of the Lie theory for Banach-Lie groups.
We recall that if an operator ideal carries a complete algebra norm that is stronger than the operator norm and for which the natural involution $T\mapsto T^*$ is continuous, then the classical groups associated with that ideal have natural structures of Banach-Lie groups. That is, studies on classical groups associated with operator ideals required the ideals to be endowed with complete norms which are moreover algebra norms (i.e., submultiplicative) and for which the involution $T\mapsto T^*$ is continuous. The submultiplicativity of the norm implies that all rank one projections, since they are unitarily equivalent, have norm precisely 1. 
From this we can easily construct ideals that lie outside this class of special complete normed ideals.

The simplest example of such an ideal that lacks a complete norm of this sort is the finite rank ideal $\Fc(\Hc)$, also a principal ideal generated by any nonzero finite rank operator.
To see that this ideal is not a complete normed ideal of this type, assume otherwise. 
Every rank one projection operator is contained in $\Fc(\Hc)$
and so also is each rank one projection and hence norm 1 operator $P_n := \diag (0, \dots , 0, 1, 0, \dots )$ 
where 1 is in the $n^{\rm th}$ position.
But then $X = \sum\limits_{n\ge 1} 2^{-n} P_n$ is an absolutely convergent series and hence converges in $\Fc(\Hc)$. 
Since this norm is stronger than the operator norm 
(i.e., dominates a constant multiple of the operator norm), 
this convergence is also in the operator norm.
But then completeness implies $X = \diag (2^{-n}) \in \Fc(\Hc)$, a contradiction because $X$ has infinite rank.
Moreover, since all nonzero ideals contain $\Fc(\Hc)$, 
this argument shows $X$ must be an operator in any complete normed ideal of this type.
For any $Y = \diag (y_n) \ge 0$ for which $y_n = O(2^{-n^2})$, 
using Calkin's ideal characteristic set characterization, it is elementary to show
that no principal ideal generated by such an operator $Y$ contains $X$ 
and hence fails to be complete normable of this type.
With a little care using Calkin's characteristic axioms, 
one can insure cardinality $c$ distinct such principal ideals. 
This holds despite that unequivalent sequences can generate in this way identical principal ideals.
See also \cite{Va89} and \cite[Banach ideals Section 4.5]{DFWW04} for additional information on examples of this type. 
\end{example}

\section{On the conjugation of Cartan subalgebras}\label{sect4} 

The aim of this last section is to discuss the following question: 

\begin{question}\label{first_quest}
\normalfont
To what extent Theorem~\ref{first_th} 
holds true when the corresponding finite-dimensional Lie groups and Lie algebras are replaced by 
$$G=U_{\Ic}(\Hc)\text{ and } \gg=\ug_{\Ic}(\Hc)$$ 
where we use the notation of Definition~\ref{real_gr}? 
\end{question}
  
The point is that here we look for conjugacy results involving the smaller unitary group~$U_{\Ic}(\Hc)$ 
rather than the full unitary group $U(\Hc)$. 
We recall that the variant of the above question with $G=U(\Hc)$ was addressed in \cite[page 33]{dlH72}
for the ideal of finite-rank operators, 
and in \cite[page 93]{dlH72}  
when $\gg=\Ic$ is one of the Schatten ideals. 
That argument based on the spectral theorem actually carries over directly 
to more general operator ideals  
and leads to the following infinite-dimensional version of Theorem~\ref{first_th}: 
\emph{if an operator ideal satisfies $\Ic\subsetneqq\Bc(\Hc)$, then 
for any two Cartan subalgebras $\Cc_1$ and $\Cc_2$ of $\Ic$ 
there exists $V\in U(\Hc)$ such that the corresponding unitary equivalence $X\mapsto VXV^*$ 
maps $\Cc_1$ onto $\Cc_2$}. 

The answer to Question~\ref{first_quest} is obvious if $\Ic=\{0\}$ and 
is also well known in the case $\Ic=\Bc(\Hc)$; 
the relevant facts are recalled in Remark~\ref{first_rem} below. 
Let us also note that problems similar to Question~\ref{first_quest} 
could be raised in connection with the other classical groups associated with operator ideals 
from Definitions \ref{complex_gr} and \ref{real_gr}, 
and more generally about various infinite-dimensional versions of the reductive Lie groups 
(see \cite{dlH72}, \cite{Al82}, and also \cite{Be11}). 

\subsection*{The maximal abelian self-adjoint subalgebras of $\Bc(\Hc)$}
We now recall the relevant facts concerning 
Question~\ref{first_quest} in the case of the operator ideal $\Ic=\Bc(\Hc)$.  
In this case we will avoid talking about Cartan subalgebras, however.  
The reason is that if the ideal $\Ic$ is equal to $\Bc(\Hc)$, then it is also a von Neumann algebra, 
and in the theory of operator algebras the name `Cartan subalgebras' is reserved for 
maximal self-adjoint subalgebras that satisfy an extra condition 
(see for instance \cite{SS08} and \cite{Re08}). 
Therefore we will not use that name throughout the next remark. 

\begin{remark}\label{first_rem}
\normalfont
For $n=0,1,2,\dots,\infty,-\infty$ let us define the probability measure space 
$$\Xi_n=\begin{cases}
[0,1] \text{ if }n=0,\\
[0,1/2]\cup\{1,\dots,n\}\text{ if }1\le n<\infty,\\
\{1,2,\dots\}\text{ if }n=\infty, \\
[0,1/2]\cup\{1,2,\dots\}\text{ if }n=-\infty
\end{cases} $$
endowed with the probability measure $\mu_n$, where 
$\mu_n$ is the Lebesgue measure for $n=0$, 
$\mu_n$ is the Lebesgue measure on $[0,1/2]$ and $\mu_n(k)=1/2n$ if $1\le k\le n<\infty$, 
moreover $\mu_\infty(k)=1/2^k$ if $k\ge 1$, 
and finally $\mu_{-\infty}$ is the Lebesgue measure on $[0,1/2]$ and $\mu_{-\infty}(k)=1/2^{k+1}$ if $1\le k<\infty$. 

Thus $L^2(\Xi_n,\mu_n)$ is a separable infinite-dimensional complex Hilbert space for $n=0,1,2,\dots,\infty,-\infty$.
If we embed $L^\infty(\Xi_n,\mu_n)$ into $\Bc(L^2(\Xi_n,\mu))$ as multiplication operators, 
then we obtain a maximal abelian self-adjoint subalgebra $\Ac_n$. 
By looking at their minimal projections, we can see that the abelian von Neumann algebras $\Ac_n$ and $\Ac_m$ are non-isomorphic to each other if $n\ne m$. 
Therefore, if $\Theta_n\colon L^2(\Xi_n,\mu_n)\to\Hc$ is any unitary operator 
and we define $\Cc_n:=\Theta_n\Ac_n\Theta_n^{-1}\subseteq\Bc(\Hc)$, 
then $\{\Cc_n\mid n=0,1,2,\dots,\infty,-\infty\}$ is a family of maximal abelian self-adjoint subalgebras of $\Bc(\Hc)$ 
which are pairwise nonisomorphic. 
Thus there exist infinitely many conjugacy classes of  
maximal abelian self-adjoint subalgebras of $\Bc(\Hc)$. 

In fact, by using Maharam's theorem on homogeneous measure algebras 
along with the basic properties of abelian von Neumann algebras 
(see \cite{Ma42} and \cite[App. IV]{Di69}), 
one can see that 
\emph{the family $\{\Cc_n\mid n=0,1,2,\dots,\infty,-\infty\}$ is 
a complete system of distinct representatives for 
the $U(\Hc)$-conjugacy classes of maximal abelian self-adjoint subalgebras of $\Bc(\Hc)$}. 
\end{remark}

\subsection*{The role of $U_{\Ic}(\Hc)$-diagonalization}\label{u-diag} 
Recall from Example~\ref{Un} that in the case of the compact linear group $U(n)$  Theorem~\ref{first_th} 
is equivalent to the fact that every skew-symmetric matrix is unitarily equivalent to a diagonal matrix. 
In view of that fact, it is easy to see that in order to be able to address Question~\ref{first_quest} 
we need to understand the set of $U_{\Ic}(\Hc)$-diagonalizable operators in $\Ic$ defined below. 

Assume that we have fixed an orthonormal basis $\bb=\{\bb_n\}_{n\ge 1}$ in $\Hc$ 
and denote by $\Dc$ the corresponding set of diagonal operators in $\Bc(\Hc)$. 
Let $\Ic$ be an operator ideal in $\Bc(\Hc)$. 
The set of 
\emph{$U_{\Ic}(\Hc)$-diagonalizable operators in $\Ic$} is
$$\Dc_{\Ic}:=\{VDV^*\mid D\in\Dc\cap\Ic,\, V\in U_{\Ic}(\Hc)\}
=\bigcup_{V\in U_{\Ic}(\Hc)}V(\Dc\cap\Ic)V^*\subseteq\Ic.$$ 
Here we have the union of the sets in the $U_{\Ic}(\Hc)$-conjugacy class of the Cartan subalgebra $\Dc\cap\Ic$ of $\Ic$. 
This is a set of normal operators in $\Ic$ and 
we will also consider its self-adjoint part 
$$\Dc_{\Ic}^{\sa}:=\Dc_{\Ic}\cap\Ic^{\sa}=\{VDV^*\mid D=D^*\in\Dc\cap\Ic,\, V\in U_{\Ic}(\Hc)\}.$$
It follows from Proposition~\ref{extremes}\eqref{extremes_item1} below  
that 
\begin{equation}\label{df}
\text{ if }\Ic=\Fc(\Hc),\text{ then }\Dc_{\Ic}^{\sa}=\Ic^{\sa}
\end{equation}
however we will prove that the latter equality fails to be true 
for any other nontrivial ideal; see Proposition~\ref{conj} below. 

The geometric shape of the set $\Dc_{\Ic}^{\sa}$ is not clear in general, 
for instance Remark~\ref{WvN} shows that it need not be a linear subspace of $\Ic^{\sa}$. 
In the case when $\Ic\ne\Ic^2$, some information on the shape and size of $\Dc_{\Ic}$ 
is discussed in Remark~\ref{info} below.

\begin{proposition}\label{extremes} 
We have the following descriptions of the set of $U_{\Ic}(\Hc)$-diag\-onalizable 
operators when $\Ic$ is the smallest or the largest operator ideal in $\Bc(\Hc)$:
\begin{enumerate}
\item\label{extremes_item1}
If $\Ic=\Fc(\Hc)$, then $\Dc_{\Ic}$ is the set of all finite-rank normal operators in $\Bc(\Hc)$. 
\item\label{extremes_item2}
If $\Ic=\Bc(\Hc)$, then $\Dc_{\Ic}$ is the set of all normal operators in $\Bc(\Hc)$ 
with pure point spectrum, in the sense that their spectral measure is supported by 
the countable subset of the spectrum consisting of the eigenvalues. 
For such a normal operator its eigenvalues are everywhere dense in its spectrum.
\end{enumerate}
\end{proposition}

\begin{proof} 
\eqref{extremes_item1}
We have already noted above that the set $\Dc_{\Fc(\Hc)}$ 
consists of finite-rank normal operators. 
Conversely, for any finite-rank normal operator $A\in\Bc(\Hc)$ 
there exist  $\lambda_1,\dots,\lambda_n\in\CC$ 
and an orthonormal system $\{v_1,\dots,v_n\}$ such that 
$A=\sum_{k=1}^n\lambda_k(\cdot,v_k)v_k$, where the integer $n\ge 1$ depends on~$A$. 
If we denote by $\Hc_0$ the linear subspace of $\Hc$ spanned by the set 
$\{v_1,\dots,v_n\}\cup\{\bb_1,\dots,\bb_n\}$, 
then $\dim\Hc_0<\infty$.  
By completing the orthonormal systems $\{v_1,\dots,v_n\}$ and $\{\bb_1,\dots,\bb_n\}$ 
to orthonormal bases in $\Hc_0$ we can see that 
there exists a unitary operator $W_0\colon\Hc_0\to\Hc_0$ 
such that $W_0v_k=\bb_k$ for $k=1,\dots,n$. 
Now let $W\colon\Hc\to\Hc$ be the unitary operator defined by the conditions 
$W\vert_{\Hc_0}=W_0$ and $Wv=v$ for every $v\in\Hc_0^\perp$. 
Then $W\in U(\Hc)\cap (\1+\Fc(\Hc))$ and moreover we have 
$$W^*AW=\sum_{k=1}^n\lambda_k(\cdot,Wv_k)Wv_k =\sum_{k=1}^n\lambda_k(\cdot,\bb_k)\bb_k\in\Dc\cap\Fc(\Hc)$$ 
hence $A\in\Dc_{\Fc(\Hc)}$. 

\eqref{extremes_item2} 
If $\Ic=\Bc(\Hc)$, then $U_{\Ic}(\Hc)=U(\Hc)$, hence 
$\Dc_{\Ic}$ is precisely the set of all operators in $\Bc(\Hc)$ which are unitary equivalent 
to diagonal operators with respect to some orthonormal basis. 
If $T\in\Bc(\Hc)$ is a diagonal operator with respect to some orthonormal basis in $\Hc$, 
then it is a normal operator and its spectral measure $E_T$ is supported on the set $\Lambda$ 
of eigenvalues, that is, 
for every Borel set $\sigma\subseteq\CC$ we have $E_T(\sigma)=E_T(\sigma\cap\Lambda)$. 
Moreover, the spectrum of $T$ is the closure of the set of eigenvalues 
(see for instance \cite[Problem 63]{Ha82}).  
These spectral properties are preserved by unitary equivalence, 
hence they are shared by every operator in~$\Dc_{\Bc(\Hc)}$.

Conversely, let $T\in\Bc(\Hc)$ be a normal operator whose spectral measure 
is supported by the set of eigenvalues $\Lambda$. 
It follows by the spectral theorem or by a direct verification that 
the eigenspaces corresponding to distinct eigenvalues of any normal operator are mutually orthogonal. 
Since the Hilbert space is separable, it then follows that the set $\Lambda$ is at most countable, 
and then the countable additivity property of the spectral measure $E_T$ 
implies that we have $\1=\sum_{\lambda\in\Lambda}E_T(\{\lambda\})$, 
where the sum of mutually orthogonal operators is convergent in the strong operator topology. 
Moreover, for every $\lambda\in\Lambda$ and $v\in\Ran E_T(\{\lambda\})$ we have $Tv=\lambda v$. 
Therefore, if we pick arbitrary orthogonal bases in the subspaces $\Ran E_T(\{\lambda\})$ for $\lambda\in\Lambda$ 
and then we take the union of these bases, then we obtain an orthogonal basis in $\Hc$ 
such that $T$ is a diagonal operator with respect to that basis, and this completes the proof. 
\end{proof}

\begin{remark}\label{eigenvalues}
\normalfont
In connection with Proposition~\ref{extremes}\eqref{extremes_item2} 
we note that if $T\in\Bc(\Hc)$ is a normal operator 
whose set of eigenvalues is everywhere dense in the spectrum, 
then $T$ need not have pure point spectrum. 
For instance, consider the unit interval $I=[0,1]\subset\RR$ endowed with the Lebesgue measure 
and define the separable Hilbert space $\Hc=\Hc_1\oplus\Hc_2$, 
where $\Hc_j=L^2(I)$ for $j=1,2$. 
Let $\{v_n\}_{n\ge 1}$ be any orthonormal basis in $\Hc_1$, 
$\{\lambda_n\}_{n\ge 1}$ be any dense sequence in the interval $I$, 
and $M\in\Bc(\Hc_2)$ be the multiplication operator defined by $(Mf)(t)=tf(t)$ for $f\in \Hc_2=L^2(I)$ and 
a.e.\ $t\in I$. 
If $T\in\Bc(\Hc)$ is the operator defined by $Tv_n=\lambda_nv_n$ for every $n\ge 1$ 
and $T\vert_{\Hc_2}=M$, then $T$ is a self-adjoint operator whose eigenvalues is everywhere dense in the spectrum, 
and yet its eigenspaces do not span the whole space, hence $T$ does not have pure point spectrum.

We refer to \cite{Wi76} for alternative characterizations of normal operators with pure point spectrum 
in the sense of Proposition~\ref{extremes}\eqref{extremes_item2} above. 
\end{remark}

\begin{remark}\label{WvN}
\normalfont
If $\Ic=\Bc(\Hc)$, then Proposition~\ref{extremes}\eqref{extremes_item2} 
shows that $\Dc_{\Ic}^{\sa}\subsetneqq\Ic^{\sa}$. 

Moreover, the Weyl-von Neumann theorem implies that if $A=A^*\in\Bc(\Hc)$,  
then there exist $A_1\in\Dc_{\Ic}^{\sa}$ and 
$A_2=A_2^*\in\Kc(\Hc)$ such that $A=A_1+A_2$ 
(see for instance~\cite[Sect. II.4]{Da96}, or \cite{Ku58} 
for a generalization involving symmetrically normed ideals). 
In particular, if the self-adjoint operator $A$ does not have pure point spectrum, 
then we obtain $A_1,A_2\in\Dc_{\Ic}^{\sa}$ and $A_1+A_2=A\not\in\Dc_{\Ic}^{\sa}$. 
This shows that $\Dc_{\Ic}^{\sa}$ fails to be a linear subspace of $\Ic^{\sa}$ if $\Ic=\Bc(\Hc)$. 
\end{remark}

Now the following question naturally arises. 

\begin{question}\label{second_quest}
\normalfont
Is it true that the conclusion of Remark~\ref{WvN} can be generalized, 
in the sense that for every nonzero ideal $\Ic$ in $\Bc(\Hc)$, 
the set $\Dc_{\Ic}^{\sa}$  is a proper subset of $\Ic^{\sa}$,  
fails to be a real linear space, and linearly spans $\Ic^{\sa}$?  
\end{question}


The first part of the above question is answered in the affirmative  
by the following statement. 
We still don't know the answer to the second part of the above question, 
except for the information derived in Proposition~\ref{July6_cor} 
and Remark~\ref{info} below. 

\begin{proposition}\label{conj}
If $\Ic$ is an operator ideal in $\Bc(\Hc)$ such that  $\Fc(\Hc)\ne\Ic$, 
then we have $\Dc_{\Ic}^{\sa}\subsetneqq\Ic^{\sa}$. 
\end{proposition}

\begin{proof}
See \cite{BPW13}. 
\end{proof}

\begin{proposition}\label{July6_cor}
There exists an injective linear mapping 
$\Pi\colon\Bc(\Hc)\to\Bc(\Hc)$ with the properties: 
\begin{enumerate}
\item\label{July6_cor_item0}
$\Pi$ is a homomorphism of triple systems, that is, 
for all $R,S,T\in\Bc(\Hc)$ we have $\Pi(RST)=\Pi(R)\Pi(S)\Pi(T)$. 
\item\label{July6_cor_item0.5} 
If $S,T\in\Bc(\Hc)$, then we have $ST=TS$ if and only if $\Pi(S)\Pi(T)=\Pi(T)\Pi(S)$. 
\item\label{July6_cor_item1} 
For every $T\in\Bc(\Hc)$ we have $\Pi(T^*)=\Pi(T)^*$. 
\item\label{July6_cor_item2} 
$\Bc(\Hc)^{+}\cap\Ran\Pi=\{0\}$. 
\item\label{July6_cor_item2.5} 
If $\Ic_1$ and $\Ic_2$ are operator ideals in $\Bc(\Hc)$, 
then we have $\Ic_1\subseteq\Ic_2$ if and only if $\Pi(\Ic_1)\subseteq\Pi(\Ic_2)$. 
 \item\label{July6_cor_item3} 
 For every operator ideal $\Ic$ in $\Bc(\Hc)$ we have $\Pi(\Ic)\subseteq\Ic$, $\Pi^{-1}(\Ic)\subseteq\Ic$, and 
$\Pi^{-1}(\Dc_{\Ic})\subseteq\Ic^2$. 
\end{enumerate}
\end{proposition}

\begin{proof}
See \cite{BPW13}. 
\end{proof}

\begin{remark}\label{info}
\normalfont
In the case when $\Ic^2\ne \Ic$, Proposition~\ref{July6_cor} is related to 
the second part of Question~\ref{second_quest} 
inasmuch as it provides some information on the size of the set $\Dc_{\Ic}$. 
We can thus get a feeling of the gap between the sets $\Dc_{\Ic}^{\sa}\subsetneqq\Ic^{\sa}$ 
referred to in Proposition~\ref{conj}. 

More precisely, it follows by Proposition~\ref{July6_cor}\eqref{July6_cor_item3} that  
$$\Dc_{\Ic}\cap\Pi(\Ic)\subseteq\Pi(\Ic^2)\subsetneqq\Pi(\Ic).$$ 
As the mapping $\Pi\colon\Ic\to\Ic$ is linear and injective, 
we obtain the infinite-dimensional linear subspace $\Pi(\Ic^{\sa})$ of $\Ic^{\sa}$ 
with the property that the linear subspace spanned by $\Dc_{\Ic}^{\sa}\cap\Pi(\Ic^{\sa})$ 
is a proper subspace of $\Pi(\Ic^{\sa})$. 
Note that $\Dc_{\Ic}^{\sa}\cap\Pi(\Ic^{\sa})\ne\{0\}$ since by \eqref{df} 
we have $\Dc_{\Fc(\Hc)}^{\sa}=\Fc(\Hc)^{\sa}$ , 
hence $\{0\}\subsetneqq\Pi(\Fc(\Hc)^{\sa})\subseteq\Dc_{\Ic}^{\sa}\cap\Pi(\Ic^{\sa})$. 

As another geometric feature, 
it follows by Proposition~\ref{July6_cor} that the aforementioned subspace $\Pi(\Ic^{\sa})$  
meets the positive cone $\Ic^{+}$ only at the vertex $0$. 
\end{remark}

\begin{example}
\normalfont
If $0< p<\infty$ and $\Ic=\Sg_p(\Hc)$ is the $p$-th Schatten ideal, then 
Proposition~\ref{July6_cor} provides some nontrivial information on the set $\Dc_{\Ic}$.    
More specifically, we have $\Ic^2=\Sg_{p/2}(\Hc)$, hence 
$\Ic^2\ne \Ic$ and the above Remark~\ref{info} applies. 
\end{example}

For the sake of completeness, let us mention that in the case when $\Ic$ is the Hilbert-Schmidt ideal, 
some sufficient conditions 
for $U_{\Ic}(\Hc)$-diagonalizability were provided in \cite{Hi85} as follows. 

\begin{theorem}
Let $\Ic=\Sg_2(\Hc)$ be the Hilbert-Schmidt ideal. 
Let $X=X^*\in\Ic$ and and denote $x_{ij}=(X\bb_j, \bb_i)$ for all $i,j\ge 1$.  
If there exist $\rho,s\in\RR$ such that $0<\rho<1$ and $0<s\le 3(1-\rho)/100$ such that 
$$(\forall j\ge1)\quad \vert x_{j+1,j+1}\vert\le\rho\vert x_{jj}\vert$$
and 
$$(\forall i,j\ge 1,\ i\ne j)\quad \vert x_{ij}\vert^2\le\frac{s^2}{(ij)^2}\cdot\vert x_{ii}x_{jj}\vert $$
then there exists $W\in U_{\Ic}(\Hc)$ such that $W^{-1}XW\in\Dc$. 
\end{theorem}

\begin{proof}
See \cite[Th. 1]{Hi85}. 
\end{proof}

\subsection*{Application to Cartan subalgebras}

As a direct application of the above results, we can prove the following statement 
which provides a partial answer to Question~\ref{first_quest}. 

\begin{proposition}\label{appl}
If $\Ic$ is a nonzero operator ideal in $\Bc(\Hc)$, 
then there exist at least two $U_{\Ic}(\Hc)$-conjugacy classes of Cartan subalgebras of~$\ug_{\Ic}(\Hc)$. 
\end{proposition}

\begin{proof}
If $\Ic=\Bc(\Hc)$, then the assertion follows by Remark~\ref{first_rem}. 
If $\{0\}\subsetneqq\Ic\subsetneqq\Bc(\Hc)$, then 
pick an orthonormal basis $\bb=\{\bb_n\}_{n\ge 1}$ in $\Hc$ 
and denote by $\Dc$ the corresponding set of diagonal operators in $\Bc(\Hc)$, 
just as above. 
Then $\Dc\cap\ug_{\Ic}(\Hc)$ is a Cartan subalgebra of $\ug_{\Ic}(\Hc)$. 
On the other hand, it follows by Proposition~\ref{conj} that 
there exists $X\in\Ic^{\sa}\setminus\Dc_{\Ic}^{\sa}$. 
Since $\ug_{\Ic}(\Hc)=i\Ic^{\sa}$, we have $iX\in\ug_{\Ic}(\Hc)$, 
and then by using Zorn's lemma we can easily find a Cartan subalgebra $\Cc$ of $\ug_{\Ic}(\Hc)$ 
with $iX\in\Cc$. 

Then the Cartan subalgebras $\Dc\cap\ug_{\Ic}(\Hc)$ and $\Cc$ of $\ug_{\Ic}(\Hc)$ 
fail to be $U_{\Ic}(\Hc)$-conjugated to each other. 
In fact, if there exists $V\in U_{\Ic}(\Hc)$ such that the transform $T\mapsto VTV^*$ 
maps $\Cc$ onto $\Dc\cap\ug_{\Ic}(\Hc)$, then $VXV^*\in\Dc$, 
and this is a contradiction with the fact that $X\not\in\Dc_{\Ic}^{\sa}$. 
This completes the proof. 
\end{proof}

The partial answer to Question~\ref{first_quest} provided in Proposition~\ref{appl} 
raises the interesting question 
of classifying the $U_{\Ic}(\Hc)$-conjugacy classes of Cartan subalgebras of $\ug_{\Ic}(\Hc)$  
for an arbitrary operator ideal $\Ic$ in $\Bc(\Hc)$. 
In this connection, we will show in \cite{BPW13} that the above Proposition~\ref{appl} 
can be considerably strengthened.  
Specifically, the forthcoming paper will contain a proof of the fact 
that \emph{if $\{0\}\subsetneqq\Ic\subsetneqq\Bc(\Hc)$, 
then there exist uncountably many $U_{\Ic}(\Hc)$-conjugacy classes of Cartan subalgebras of~$\ug_{\Ic}(\Hc)$}. 
This result is strikingly different from the situation $\Ic=\Bc(\Hc)$, 
where one has only countably many conjugacy classes of Cartan subalgebras, 
as noted in Remark~\ref{first_rem} above. 

To conclude we summarize the information available so far on the $U_{\Ic}(\Hc)$-con\-ju\-gacy classes 
of Cartan subalgebras or maximal abelian self-adjoint subalgebras 
of~$\ug_{\Ic}(\Hc)$, where $\Ic\subseteq\Bc(\Hc)$ is a nonzero operator ideal: 
\begin{itemize}
\item if $\dim\Hc<\infty$ then there is precisely one conjugacy class (Theorem~\ref{first_th}); 
\item if $\dim\Hc=\infty$ and $\Ic=\Bc(\Hc)$ then there are countably many conjugacy classes (Remark~\ref{first_rem}); 
\item if $\dim\Hc=\infty$ and $\{0\}\subsetneqq\Ic\subsetneqq\Bc(\Hc)$, 
then there are uncountably many conjugacy classes (\cite{BPW13}). 
\end{itemize}




\end{document}